\documentclass[12pt]{amsart}
\usepackage[T1]{fontenc}
\usepackage[latin1]{inputenc}
\usepackage{typearea}
\usepackage{geometry}
\usepackage{ulem}
\usepackage{amsmath}
\usepackage{amssymb}
\usepackage{latexsym}
\usepackage{enumerate}
\usepackage{amsthm}
\usepackage[all]{xy}
\usepackage{hhline}
\usepackage{epsf} 
\usepackage{cite}

\newtheorem{theorem}{{\sc Theorem}}[section]
\newtheorem{cor}[theorem]{{\sc Corollary}}

\newtheorem{lemma}[theorem]{{\sc Lemma}}
\newtheorem{prop}[theorem]{{\sc Proposition}}

\theoremstyle{remark}
\newtheorem{remark}[theorem]{{\sc Remark}}

\theoremstyle{definition}

\newcommand{\R}{\mathbb{R} }

\newcommand{\B}{\mathcal{B}}
\newcommand{\F}{\mathcal{F}}

\newcommand{\D}{\mathcal{D}}
\newcommand{\K}{\mathcal{K}}
\newcommand{\Z}{\mathbb{Z}}

\newcommand{\calL}{\mathcal{L}}

\newcommand{\calw}{\mathcal{W}}

\newcommand{\ga}{\gamma}

\providecommand{\abs}[1]{\lvert #1\rvert}

\providecommand{\fnorm}[1]{\lVert #1\rVert_\infty}

\DeclareMathOperator{\Lip}{Lip}

\renewcommand{\phi}{\varphi}
\renewcommand{\epsilon}{\varepsilon}
\newcommand{\eps}{\varepsilon}
\renewcommand{\rho}{\varrho}

\begin{document}
\title[Rate of convergence for the arcsine law ]{A rate of convergence for the arcsine law by Stein's method}
\author{Christian D\"obler}
\thanks{The author has been supported by Deutsche Forschungsgemeinschaft via SFB-TR 12.\\
Ruhr-Universit\"at Bochum, Fakult\"at f\"ur Mathematik, NA 3/68, D-44780 Bochum, Germany. \\
christian.doebler@ruhr-uni-bochum.de }
\begin{abstract}
Using Stein's method for the Beta distributions and a recent technique by Goldstein and Reinert of comparing the Stein characterization of the target distribution with that of the approximating distribution we prove a rate of convergence in the classical arcsine law, which states that the distribution of the relative time spent positive by a symmetric random walk on $\Z$ converges weakly to the arcsine distribution on $[0,1]$. 
\end{abstract}

\maketitle
\section{Introduction}
Consider a game between two players $A$ and $B$ that consists of consecutive tossings of a fair coin. Each time the coin shows heads, player $A$ has to pay one dollar to player $B$ and conversely, each time the coin shows tails, player $A$ obtains one dollar from player $B$. If we consider the process that gives for each discrete time $n$ the current fortune of player $A$, then by the symmetry of the model one is led to the conjecture, that for $n$ sufficiently large, the relative amount of time, that player $A$ is in the lead should be roughly one-half. The so-called (first) arcsine law states, that this intuition is entirely wrong. In fact, it is more likely that one of the players will lead for nearly all of the time.\\ 
Let $(S_k)_{k\geq0}$ be the symmetric random walk on $\Z$, i.e. we have\\ $S_k:=\sum_{j=1}^k\eps_j$ ($k\geq0$) for i.i.d. random variables $\eps_1,\eps_2,\ldots$ with $P(\eps_1=1)=P(\eps_1=-1)=\frac{1}{2}$. Letting $X_j:=1_{\{S_{j-1}\geq0,\,S_j\geq0\}}$, 
$T_m:=\sum_{j=1}^{2m} X_j$, $R_m:=\frac{1}{2}T_m$ and $W_m:=\frac{1}{2m}T_m=\frac{1}{m}R_m$ it is a classical result, first proven by Paul L\'{e}vy for Brownian motion, that as $m\to\infty$ we have 

\[\calL(W_m)\stackrel{\D}{\rightarrow}\nu_{\frac{1}{2},\frac{1}{2}}\,,\] 

where, for $a,b>0$ $\nu_{a,b}$ denotes the \textit{Beta distribution} to the parameters $a$ and $b$ on $[0,1]$, which has density $q_{a,b}(x):=\frac{1}{B(a,b)}x^{a-1}(1-x)^{b-1}1_{(0,1)}(x)$ with $B(a,b)$ denoting the \textit{Beta function}. Consequently, $\nu:=\nu_{\frac{1}{2},\frac{1}{2}}$ is the arcsine distribution on $[0,1]$ with density  $q(x):=q_{\frac{1}{2},\frac{1}{2}}(x)=\frac{1}{\pi}\frac{1}{\sqrt{x(1-x)}}1_{(0,1)}(x)$. A proof of this well-known theorem can be found for example in \cite{Fel1}.\\
 Recently, there has been some progress in Stein's method for the family of Beta distributions by Goldstein and Reinert (see \cite{GolRei12}) and by the author of the present article (see \cite{Doe12}). Both of these preprints prove rates of convergence in a Polya urn model. In this paper we will use the general results from the preprints \cite{GolRei12} and \cite{Doe12} and especially the technique of comparing the Stein characterization of the target distribution with that of the (discrete) approximating distribution from \cite{GolRei12} to prove bounds on the Wasserstein distance of the distribution $\calL(W_m)$ of $W_m$ to the arcsine distribution $\nu$.

\section{Stein's method for the arcsine distribution and for $\calL(R_m)$}
Stein's method is very useful tool for proving distributional convergence. Its main advantage over other techniques is, that it automatically yields concrete error bounds on various distributional distances. Since its introduction in 1972 in the seminal paper \cite{St72} by Charles Stein for the univariate standard normal distribution, there has been much progress in adapting Stein's idea of linking a characterizing operator for the target distribution to a differential equation (in the absolutely continuous case) or to a difference equation (in the discrete case), the \textit{Stein equation}, to other distributions, as for example the Poisson distribution (see \cite{Ch75}), the Gamma distribution (see \cite{Luk}), the exponential distribution (see \cite{CFR11} and \cite{PekRol11}), the geometric distribution (see \cite{PRR}) and many others. For a general introduction to Stein's method we refer to the book \cite{CGS} which emphasizes normal approximation but also treats approximation by other distributions. Here, we will make use of the recent development of Stein's method for the Beta distributions (specialized to the arcsine distribution), which was done independently and with different emphases in \cite{GolRei12} and \cite{Doe12}. Furthermore, we will heavily use the approach from \cite{GolRei12} for finding various Stein characterizations of a distribution supported on $\Z$. This topic was also explored in \cite{LeSwd}.\\
We start with Stein's method for the arcsine distribution $\nu$. The following result, a slight variant of Proposition 3.1 in \cite{Doe12}, gives a Stein characterization for $\nu$. We denote by $\K$ the class of all continuous and piecewise continuously differentiable functions $f:\R\rightarrow\R$ vanishing at infinity with $\int_\R\abs{f'(x)}\sqrt{x(1-x)}dx<\infty$.

\begin{prop}\label{chararcsin}
A real-valued random variable $X$ has the arcsine distribution $\nu$ on $[0,1]$ if and only if for all functions $f\in\K$ the expected values $E\bigl[X(1-X)f'(X)\bigr]$ and $E\bigl[(X-\frac{1}{2})f(X)\bigr]$ exist and coincide. 
\end{prop}

According to Stein's idea, by Proposition \ref{chararcsin}, for a given $\nu$-integrable test function $h:\R\rightarrow\R$ one is led to consider the Stein equation 

\begin{equation}\label{steineqarcsin}
x(1-x)f'(x)+\Bigl(\frac{1}{2}-x\Bigr)f(x)=h(x)-\nu(h)\,,
\end{equation}  

which is to be solved for the unknown function $f$ of $x\in\R$ (or, at least, $x\in[0,1]$).
From the theory in Section 3 of \cite{Doe12} we know that there exists a unique solution $f_h$ to (\ref{steineqarcsin}), defined on $\R$, which is bounded on $[0,1]$. For $x\in(0,1)$ it is given by 

\begin{equation}\label{steinsolarcsin}
f_h(x)=\frac{1}{x(1-x)q(x)}\int_0^x\bigl(h(t)-\nu(h)\bigr)q(t)dt=\frac{-1}{x(1-x)q(x)}\int_x^1\bigl(h(t)-\nu(h)\bigr)q(t)dt\,.
\end{equation} 

Furthermore, $f_h$ is continuous at $0$ and $1$ as long as $h$ is.
The following result is a special case of Proposition 3.7 in \cite{Doe12}.

\begin{lemma}\label{bounds}
Let $h:\R\rightarrow\R$ be Borel-measurable and $\nu$-integrable.
\begin{enumerate}[{\normalfont (a)}]
\item If $h$ is bounded, then $\fnorm{f_h}\leq2\fnorm{h-\nu(h)}$.
\item If $h$ is Lipschitz, then $\fnorm{f_h}\leq2\fnorm{h'}$ and $\fnorm{f_h'}\leq C_1\fnorm{h}$, where the finite constant $C_1$ does not depend on $h$.
\item If $h$ is twice differentiable with bounded first and second derivative, then \\
$\fnorm{f_h''}\leq C_2\bigl(\fnorm{h'}+\fnorm{h''}\bigr)$, where the finite constant $C_2$ does not depend on $h$.  
\end{enumerate}
\end{lemma}

\begin{proof}
Noting that $\frac{1}{2}$ is the median for $\nu$ with $q(1/2)=\frac{2}{\pi}$, this follows immediately from Proposition 3.7 in \cite{Doe12} with $a=b=\frac{1}{2}$.
\end{proof}

\begin{remark}
Note that the bounds on the Stein solutions $f_h$ from Lemmas 3.3-3.5 in \cite{GolRei12} do not cover the case of the arcsine distribution, since they all impose the condition $a,b\in[1,\infty)$. It is the restriction to these parameters, that allows the authors to provide explicit constants in place of $C_1$ from Lemma \ref{bounds} which furthermore yield explicit constants for the rate of convergence in the Polya urn model. 
\end{remark}

Now, we turn to a suitable version of Stein's method for the distribution $\calL(R_m)$ of $R_m$. To do so, we will first give an explicit formula for the probability mass function $p(k)$ of $R_m$, which is given by a famous Theorem by Chung and Feller.
First, however, we will need the following lemma.

\begin{lemma}\label{lemxj}
Let  $m$ be a positive integer. Then 
\begin{enumerate}[{\normalfont (a)}]
 \item $X_{2j-1}=X_{2j}$ for $j=1,\ldots,m$
 \item $T_{m}$ has values in $2\cdot\{0,\ldots,m\}$ and hence $R_m$ has values in $\{0,\ldots,m\}$.
 \item Letting $\tilde{X}_j:=1-X_j$ we have $\tilde{X}_j=1_{\{S_{j-1}\leq0,\,S_j\leq0\}}$ and\\ $(\tilde{X}_1,\ldots,\tilde{X}_n)\stackrel{\D}{=}(X_1,\ldots,X_n)$.
\end{enumerate}
\end{lemma}
\begin{proof}
We prove (a) by induction on $j$. It is easy to see, that $X_1=X_2$ always holds. Now let $1\leq j\leq m-1$. Then we have 
$X_{2j-1}=X_{2j}$ by the induction hypothesis. Suppose, that $X_{2j-1}=X_{2j}=1$. If $S_{2j}=0$, then the claim $X_{2j+1}=X_{2j+2}$ follows in the same manner as $X_1=X_2$. If $S_{2j}>0$, then necessarily $S_{2j}\geq2$, yielding $S_{2j+1}>0$ and $S_{2j+1}\geq0$. Hence, $X_{2j+1}=X_{2j}=1$ in this case. If, contrarily, $X_{2j-1}=X_{2j}=0$, the proof is similar.\\
Assertion (b) follows immediately from (a). The first assertion from (c) is clear since either both, $S_{2j-1}$ and $S_{2j}$, are nonnegative or nonpositive. Now, observe, that there is a (measurable) function $f$ such that $(X_1,\ldots,X_n)=f(S_1,\ldots,S_n)$. Since 
$\tilde{X}_j=  1_{\{-S_{j-1}\geq0,\,-S_j\geq0\}}$ and $(S_1,\ldots,S_n)\stackrel{\D}{=}(-S_1,\ldots,-S_n)$ by symmetry, we have 

\[(\tilde{X}_1,\ldots,\tilde{X}_n)= f(-S_1,\ldots,-S_n)\stackrel{\D}{=}f(S_1,\ldots,S_n)=(X_1,\ldots,X_n)\,.\]
\end{proof}

The proof of the following well-known theorem can be found for example in \cite{Fel1}. 

\begin{theorem}[Chung-Feller Theorem]\label{chungfeller}
Let $m$ be a positive integer. Then, for each $0\leq k\leq m$ we have 
\[P(R_m=k)=P(T_m=2k)=u_{2k}u_{2m-2k}\,,\]

where $u_0:=1$ and $u_{2j}:=\binom{2j}{j}2^{-j}$ for $j\geq1$ denotes the probability that the symmetric random walk returns to zero at time $2j$.
\end{theorem}

Thus, by Theorem \ref{chungfeller} the probability mass function $p:\Z\rightarrow\R$ corresponding to $R_m$ is given by 
$p(k)=0$ for $k\in\Z\setminus\{0,\ldots,m\}$ and by

\begin{equation}\label{formelp}
p(k)=u_{2k}u_{2m-2k}=2^{-2m}\binom{2k}{k}\binom{2m-2k}{m-k}
\end{equation}

for $0\leq k\leq m$.\\

In the following, we review the recent adaption of the so-called \textit{density approach} for absolutely continuous distributions (see, e.g. \cite{DHRS}, \cite{EiLo10}, \cite{ChSh} and \cite{CGS}) to discrete distributions on the integers, which was done in \cite{GolRei12} and also in \cite{LeSwd}. For reasons of simplicity we restrict ourselves to the case of finite integer intervals. \\
A finite integer interval is a set $I$ of the form $I=[a,b]\cap\Z$ for some integers $a\leq b$. Given a probability mass function $p:\Z\rightarrow\R$ with $p(k)>0$ for $k\in I$ and $p(k)=0$ for $k\in\Z\setminus I$, we consider the function 
$\psi:I\rightarrow\R$ given by the formula

\begin{equation}\label{formelpsi}
\psi(k):=\frac{\Delta p(k)}{p(k)},
\end{equation} 

where for a function $f$ on the integers $\Delta f(k):=f(k+1)-f(k)$ denotes the \textit{forward difference operator}. 
Note that by definition always $\psi(b)=-1$ if $I=[a,b]\cap\Z$. For such a probability mass function $p$ with support a finite integer interval $I=[a,b]\cap\Z$, let $\F(p)$ denote the class of all real-valued functions $f$ on $\Z$ such that $f(a-1)=0$.
The following result is a special case of Proposition 2.1 of \cite{GolRei12}.

\begin{prop}\label{golreidens}
Let $Z$ be a $\Z$-valued random variable with probability mass function $p$ which is supported on the finite integer interval 
$I=[a,b]\cap\Z$ and is positive there. Then, a given random variable $X$ with support $I$ has the probability mass function $p$ if and only if for all $f\in\F(p)$ it holds that

\begin{equation}\label{steiniddens}
E\bigl[\Delta f(X-1)+\psi(X)f(X)\bigr]=0\,.
\end{equation}
\end{prop}

The next result, a version of Corollary 2.1 from \cite{GolRei12}, yields various other Stein characterizations for the distribution corresponding to $p$ from Proposition \ref{golreidens}.

\begin{cor}\label{golreingen}
Let $Z$ be a $\Z$-valued random variable with probability mass function $p$ which is supported on the finite integer interval 
$I=[a,b]\cap\Z$ and is positive there. Let $c:[a-1,b]\cap\Z\rightarrow\R\setminus\{0\}$ be an arbitrary function. Then, in order that a given random variable $X$ with support $I$ has the probability mass function $p$ it is necessary and sufficient that for all functions $f\in\F(p)$ we have

\begin{equation}\label{steinidgen}
E\Bigl[c(X-1)\Delta f(X-1)+\bigl[c(X)\psi(X)+\Delta c(X-1)\bigr]f(x)\Bigr]=0\,.
\end{equation}
\end{cor}

\begin{remark}
Letting $\ga(k):=c(k)\psi(k)+\Delta c(k-1)$ we see that $c$ satisfies the difference equation 
$\Delta c(k-1)=\ga(k)-c(k)\psi(k)$. This exactly corresponds to the differential equation 
$\eta'(x)=\ga(x)-\eta(x)\psi(x)$ from Formula (14) in \cite{Doe12}, where $\psi(x):=\frac{p'(x)}{p(x)}$ is the logarithmic derivative of the density $p$. In \cite{Doe12} it is shown, that this differential equation must hold, in order that a given distribution $\mu$ with density $p$ satisfies the Stein identity $E[\eta(Z)g'(Z)+\ga(Z)f(Z)]=0$, where $Z\sim\mu$.
So also in this respect, there is a strong analogy between the absolutely continuous and the discrete case.
\end{remark}

Now, with the abstract results at hand, we return to the concrete distribution of $R_m$ which has probability mass function $p$ supported on $I:=[0,m]\cap\Z$ and given by (\ref{formelp}). Using the relation 

\[\binom{2j}{j}=\frac{4j-2}{j}\binom{2j-2}{j-1} \quad\text{for }j\geq1\]

it can easily be checked, that in this case $\psi$ is given by

\begin{equation}\label{psip}
\psi(k)=\frac{2k-m+1}{\bigl(k+1\bigr)\bigl(2(m-k)-1\bigr)}\,,\quad 0\leq k\leq m\,.
\end{equation} 

This motivates the definition $c(k):=\bigl(k+1\bigr)\bigl(2(m-k)-1\bigr)$ for $k=0,\ldots,m$. 
These observations lead to the following lemma.

\begin{lemma}\label{steinidp}
Let $p$ be the probability mass function of $R_m$ as given by (\ref{formelp}). A random variable $X$ with support 
$I=[0,m]\cap\Z$ has probability mass function $p$ if and only if for all $f\in\F(p)$ it holds that 

\begin{equation}\label{steinidpeq}
E\Bigl[X\Bigl(\bigl(m-X\bigr)+\frac{1}{2}\Bigr)\Delta f(X-1)+\Bigl(\frac{m}{2}-X\Bigr)f(X)\Bigr]=0\,. 
\end{equation} 
\end{lemma}

\begin{proof}
This follows from Corollary \ref{golreingen}, since by the definition of $c$ we have 

\[c(k)\psi(k)+\Delta c(k-1)=2\Bigl(\frac{m}{2}-k\Bigr)\]

and

\[c(k-1)=k\bigl(2(m-k)+1\bigr)=2k\Bigl(m-k+\frac{1}{2}\Bigr)\,.\]

\end{proof}

\section{A rate of convergence for the arcsine law}
In this section we will use the tools from Section 2 to prove a rate of convergence in the Wasserstein distance for the arcsine law. Recall, that for two distributions $\mu_1$ and $\mu_2$ on $(\R,\B)$, whose first moments exist, the Wasserstein distance is given by 

\[d_{\calw}(\mu_1,\mu_2):=\sup_{h\in\Lip(1)}\Biggl|\int_\R hd\mu_1-\int_\R hd\mu_2\Biggr|\]

and for two real-valued random variables $X$ and $Y$ one defines 

\[d_{\calw}(X,Y):=d_{\calw}\bigl(\calL(X),\calL(Y)\bigr)=\sup_{h\in\Lip(1)}\abs{E\bigl[h(X)\bigr]-E\bigl[h(Y)\bigr]}\,,\]

where $\Lip(1)$ denotes the class of all Lipschitz-continuous functions $h$ on $\R$ with minimal Lipschitz constant $\fnorm{h'}\leq1$. It is known, that on the space of probability measures with existing first moment, convergence in the Wasserstein distance is stronger than weak convergence.

\begin{theorem}\label{maintheo}
There exists a finite constant $C>0$ such that for each positive integer $m$ we have 

\[d_{\calw}\bigl(\calL(W_m),\nu\bigr)\leq\frac{C}{m}\,.\]

\end{theorem}

\begin{remark}
To the best of my knowledge this is the first result that gives a rate of convergence of order $m^{-1}$ for the arcsine law. 
The restriction to even times $n=2m$ is immaterial and only for convenience, since the formula from the Chung-Feller Theorem only holds for these times. Since for odd times $n=2m+1$

\[\Bigl|\frac{1}{2m+1}\sum_{j=1}^{2m+1}X_j-W_m\Bigr|\leq\frac{2}{2m+1}\,,\]

the same rate of convergence also holds for the whole sequence of positive times of the random walk.\\
It may be seen from the proof of Theorem \ref{maintheo}, that the constant $C$ can be made explicit in terms of the constant $C_1$ from Lemma \ref{bounds}. 
\end{remark}

\begin{proof}[Proof of Theorem \ref{maintheo}]
The proof follows the lines of the proof of Theorem 3.1 from \cite{GolRei12} and is included for reasons of completeness. Using the notation from \cite{GolRei12}, for a function $f$ and $y>0$ let

\[\Delta_y f(x):=f(x+y)-f(x)\,.\]

We also write $W:=W_m$ and $R:=R_m$. Let $h\in\Lip(1)$ be fixed and let $f:=f_h$ be the corresponding solution to the Stein equation (\ref{steineqarcsin}) given by (\ref{steinsolarcsin}) for $x\in(0,1)$ but which we set equal to zero for $x\in\R\setminus[0,1]$. Consider the function $g(x):=f(x/m)$ which is zero on $\R\setminus[0,m]$. Then by Lemma \ref{steinidp} and upon dividing by $m$ in (\ref{steinidpeq}) we obtain

\begin{eqnarray*}
0&=&\frac{1}{m}E\Bigl[R\Bigl(\bigl(m-R\bigr)+\frac{1}{2}\Bigr)\Delta g(R-1)+\Bigl(\frac{m}{2}-R\Bigr)g(R)\Bigr]\\
&=&E\Bigl[mW\Bigl(\bigl(1-W\bigr)+\frac{1}{2m}\Bigr)\Delta_{1/m}f\Bigl(W-\frac{1}{m}\Bigr)+\Bigl(\frac{1}{2}-W\Bigr)f(W)\Bigr]
\end{eqnarray*}

Inserting this into the Stein identity resulting from the stein equation (\ref{steineqarcsin}) we obtain

\begin{eqnarray}\label{eq1}
&&E\bigl[h(W)\bigr]-\nu(h)\nonumber\\
&=&E\Bigl[W(1-W)f'(W)+\Bigl(\frac{1}{2}-W\Bigr)f(W)\Bigr]\nonumber\\
&=&E\Bigl[W(1-W)f'(W)-mW\Bigl(\bigl(1-W\bigr)+\frac{1}{2m}\Bigr)\Delta_{1/m}f\Bigl(W-\frac{1}{m}\Bigr)\Bigr]\nonumber\\
&=&E\Bigl[W(1-W)f'(W)-mW\bigl(1-W\bigr)\Delta_{1/m}f\Bigl(W-\frac{1}{m}\Bigr)\Bigr]+E_1\,,
\end{eqnarray}

with 

\begin{equation}\label{e1}
\abs{E_1}=\frac{1}{2}\Bigl|E\Bigl[W\Delta_{1/m}f\Bigl(W-\frac{1}{m}\Bigr)\Bigr]\Bigr|
\leq\frac{1}{2m}\fnorm{f'}E[W]=\frac{1}{4m}\fnorm{f'}\leq\frac{C_1}{4m}
\end{equation}

by Lemma \ref{bounds} (b) and by $E[W]=1/2$, which follows for example from Lemma \ref{lemxj} (c) by symmetry.
Using the fundamental theorem of calculus, we rewrite the remaining term in (\ref{eq1}) as

\begin{eqnarray}\label{eq2}
&& E\Bigl[W(1-W)f'(W)-mW\bigl(1-W\bigr)\Delta_{1/m}f\Bigl(W-\frac{1}{m}\Bigr)\Bigr]\nonumber\\
&=&E\Biggl[W(1-W)\Biggl(f'(W)-m\int_{W-\frac{1}{m}}^Wf'(t)dt\Biggr)\Biggr]\nonumber\\
&=&mE\Biggl[\int_{W-\frac{1}{m}}^W W(1-W)\Bigl(f'(W)-f'(t)\Bigr)dt\Biggr]\nonumber\\
&=&mE\Biggl[\int_{W-\frac{1}{m}}^W\Bigl(W(1-W)f'(W)-t(1-t)f'(t)\Bigr)dt\Biggr] +E_2
\end{eqnarray}

where 

\[E_2:=-mE\Biggl[\int_{W-\frac{1}{m}}^W\Bigl(W(1-W)-t(1-t)\Bigr)f'(t)dt\Biggr]\]

and hence, again by the fundamental theorem of calculus, 

\begin{eqnarray}\label{e2}
\abs{E_2}&=&m\Biggl|E\Biggl[ \int_{W-\frac{1}{m}}^Wf'(t)\int_t^W(1-2s)dsdt\Biggr]\Biggr|\nonumber\\
&\leq&m\bigl(1+\frac{2}{m}\bigr)\fnorm{f'}E\Biggl[\int_{W-\frac{1}{m}}^W(W-t)dt\Biggr]\nonumber\\
&=&(m+2)\fnorm{f'}\int_0^{\frac{1}{m}}udu\nonumber\\
&=&\frac{m+2}{2m^2}\fnorm{f'}\leq C_1\frac{m+2}{2m^2}
\end{eqnarray}

where we have used Lemma \ref{bounds} for the last step and the inequality $\abs{1-2s}\leq\bigl(1+\frac{2}{m}\bigr)$ for relevant values of $s$ for the first step.\\
It remains to deal with the first expectation in (\ref{eq2}). Since $f=f_h$ solves the Stein equation (\ref{steineqarcsin}) and by the fundamental theorem of calculus for Lebesgue integration, we obtain 

\begin{eqnarray}\label{eq3}
&&mE\Biggl[\int_{W-\frac{1}{m}}^W\Bigl(W(1-W)f'(W)-t(1-t)f'(t)\Bigr)dt\Biggr] \\
&=&mE\Biggl[\int_{W-\frac{1}{m}}^W\Bigl(h(W)-\nu(h)-\Bigl(\frac{1}{2}-W\Bigr)f(W)
-h(t)+\nu(h)+\Bigl(\frac{1}{2}-t\Bigr)f(t)\Bigl)dt\Biggr]\nonumber\\
&=&m E\Biggl[\int_{W-\frac{1}{m}}^W\Bigl(h(W)-h(t)+\Bigl(W-\frac{1}{2}\Bigr)f(W)-\Bigl(t-\frac{1}{2}\Bigr)f(t)\Bigl)dt\Biggr]\nonumber\\
&=&m E\Biggl[\int_{W-\frac{1}{m}}^W\biggl(\int_t^W h'(s)ds+\int_t^W\Bigl(f(s)+\Bigl(s-\frac{1}{2}\Bigr)f'(s)\Bigr)ds\biggr)dt\Biggr]\nonumber
\end{eqnarray}

The inner integrals from (\ref{eq3}) are bounded separately. As to the first one, 

\begin{equation}\label{e3}
 \Biggl|mE\Biggl[\int_{W-\frac{1}{m}}^W\int_t^W h'(s)dsdt\Biggr]\Biggr|\leq m\fnorm{h'}E\Biggl[\int_{W-\frac{1}{m}}^W(W-t)dt\Biggr]
=\frac{1}{2m}\fnorm{h'}\leq\frac{1}{2m}\,.
\end{equation}

For the second one, since $\abs{\frac{1}{2}-s}\leq\frac{1}{m}+\frac{1}{2}\leq3/2$ for the relevant values of $s$, we have

\begin{eqnarray}\label{e4}
&&m\Biggl|E\Biggl[\int_{W-\frac{1}{m}}^W \int_t^W\Bigl(f(s)+\Bigl(s-\frac{1}{2}\Bigr)f'(s)\Bigr)dsdt\Biggr]\Biggr|\nonumber\\
&\leq&m\Bigl(\fnorm{f}+\frac{3}{2}\fnorm{f'}\Bigr)E\Biggl[\int_{W-\frac{1}{m}}^W\int_t^Wdsdt\Biggr]\nonumber\\
&=&m\Bigl(\fnorm{f}+\frac{3}{2}\fnorm{f'}\Bigr)E\Biggl[\int_{W-\frac{1}{m}}^W (W-t)dt\Biggr]\nonumber\\
&=&m\Bigl(\fnorm{f}+\frac{3}{2}\fnorm{f'}\Bigr)\frac{1}{2m^2}\nonumber\\
&\leq&\Bigl(2+\frac{3}{2}C_1\Bigr)\fnorm{h'}\frac{1}{2m}\leq\frac{4+3C_1}{4m}\,,
\end{eqnarray}

where we have used Lemma \ref{bounds} for the next to last inequality. Since $h\in\Lip(1)$ was arbitrary, the conclusion of the theorem follows from (\ref{e1}), (\ref{e2}), (\ref{e3}) and (\ref{e4}).
\end{proof}

\normalem
\bibliography{arcsin}{}
\bibliographystyle{alpha}

\end{document}